\documentclass{article}
\usepackage{graphicx} % Required for inserting images

\usepackage{amssymb}
\usepackage{amsmath}
\usepackage{bbm}
\usepackage{geometry}
\usepackage{authblk}
\usepackage{amsthm}
\usepackage{hyperref}
\hypersetup{colorlinks=true,linkcolor=blue,citecolor=red}

\theoremstyle{plain}
\newtheorem{theorem}{Theorem}[section]

\newtheorem{lemma}{Lemma}[section]

\theoremstyle{definition}

\newtheorem{claim}{Claim}[section]

\newtheorem{case}{Case}
\newtheorem{question}{Question}[section]
\newtheorem{example}{Example}[section]

\theoremstyle{remark}
\newtheorem{remark}{Remark}[section]

\title{The Dyn-Farkhi conjecture and the convex hull of a sumset in two dimensions}
\author{Mark Meyer}

\begin{document}
\maketitle
\begin{abstract}
    For a compact set $A$ in $\mathbb{R}^n$ the Hausdorff distance from $A$ to $\textup{conv}(A)$ is defined by 
    \begin{equation*}
        d(A):=\sup_{a\in\textup{conv}(A)}\inf_{x\in A}|x-a|,
    \end{equation*}
    where for $x=(x_1,\dots,x_n)\in\mathbb{R}^n$ we use the notation $|x|=\sqrt{x_1^2+\dots+x_n^2}$. It was conjectured in 2004 by Dyn and Farkhi that $d^2$ is subadditive on compact sets in $\mathbb{R}^n$. In 2018 this conjecture was proved false by Fradelizi et al. when $n\geq3$. The conjecture can also be verified when $n=1$. In this paper we prove the conjecture when $n=2$ and in doing so we prove an interesting representation of the sumset $\textup{conv}(A)+\textup{conv}(B)$ for full dimensional compact sets $A,B$ in $\mathbb{R}^2$.
\end{abstract}

\section{Introduction}

For a compact set $A$ in $\mathbb{R}^n$ define the Hausdorff distance from $A$ to $\textup{conv}(A)$ by
\begin{equation*}
    d(A):=\sup_{a\in \textup{conv}(A)}\inf_{x\in A}|x-a|,
\end{equation*}
where for $x=(x_1,\dots,x_n)\in\mathbb{R}^n$ we use the notation $|x|=\sqrt{x_1^2+\dots+x_n^2}$. It was conjectured by Dyn and Farkhi \cite{dyn1} that for compact sets $A$ and $B$ in $\mathbb{R}^n$,
\begin{equation*}
    d^2(A+B)\leq d^2(A)+d^2(B).
\end{equation*}
For a couple of reasons it is natural to make this conjecture. In \cite{cassels1} Cassels studied the effective standard deviation, defined by
\begin{equation*}
    v^2(A)=\sup_{x\in\textup{conv}(A)}\inf\{\sum p_i|a_i-x|^2:x=\sum p_ia_i;p_i>0;\sum p_i=1,a_i\in A\}.
\end{equation*}
It was observed that $v^2$ is subadditive:
\begin{theorem}[Cassels \cite{cassels1}]\label{cassels_theorem}
    Let $A$ and $B$ be compact sets in $\mathbb{R}^n$. Then
    \begin{equation*}
        v^2(A+B)\leq v^2(A)+v^2(B).
    \end{equation*}
\end{theorem}
In the paper \cite{wegmann1} Wegmann observed that $v(A)=d(A)$ if the supremum in the definition of $v(A)$ is achieved in the relative interior of $\textup{conv}(A)$. Then for such compact sets $A$ and $B$ we have $d^2(A+B)\leq d^2(A)+d^2(B)$. Another inequality of interest follows from the Shapley-Folkman-Starr theorem \cite{starr1,starr2}, which shows that in the Hausdorff metric large sums tend towards convex sets (as long as the diameters of the sets are bounded). For compact $A\subseteq\mathbb{R}^n$, define the radius of the smallest ball containing $A$ by
\begin{equation*}
    \textup{rad}(A):=\inf_{x\in\mathbb{R}^n}\sup_{a\in A}|x-a|.
\end{equation*}
The following theorem relates the functions $d(A)$ and $\textup{rad}(A)$.
\begin{theorem}[Starr \cite{starr1}]\label{starr_theorem}
Let $A$ and $B$ be compact sets in $\mathbb{R}^n$. Then
\begin{equation*}
    d^2(A+B)\leq \textup{rad}^2(A)+\textup{rad}^2(B).
\end{equation*}
\end{theorem}
We note that Starr actually proved a better result for large sums, and Theorem \ref{starr_theorem} is the special case where only two sets are added together. In the same paper it was observed that $d(A)\leq \textup{rad}(A)$ for any compact $A$. With these results in mind, it seems obvious that $d^2$ should be subadditive. It turns out that this is not true. In \cite{fradelizi1} Fradelizi et al. proved a counter example when $n\geq3$:
\begin{theorem}[Fradelizi et al. \cite{fradelizi1}]
    Let $q\geq0$. The inequality
    \begin{equation*}
        d^q(A+B)\leq d^q(A)+d^q(B)
    \end{equation*}
    holds for all compact sets $A,B\subseteq\mathbb{R}^3$ if and only if $q\leq 1$.
\end{theorem}
We emphasize that although the conjecture was proved false, it is still an open problem to determine if the conjecture is true when $A=B$. The counter example was constructed by adding together $2$-dimensional sets in $\mathbb{R}^3$ whose sum is a three dimensional set, so such an example could not be constructed in $\mathbb{R}^2$. In this paper we will prove that the conjecture is true in $\mathbb{R}^2$.

\begin{theorem}\label{dyn_farkhi_2_dimension}
    Let $A$ and $B$ be compact sets in $\mathbb{R}^2$. Then
    \begin{equation*}
        d^2(A+B)\leq d^2(A)+d^2(B).
    \end{equation*}
    This is the best upper bound for $d(A+B)$ in terms of $d(A)$ and $d(B)$ in the following sense: If $d_A$ and $d_B$ are non-negative real numbers, then there exist compact sets $A$ and $B$ in $\mathbb{R}^2$ such that $d(A)=d_A$, $d(B)=d_B$, and $d^2(A+B)=d_A^2+d_B^2$.
\end{theorem}
As observed in the above theorem, the bound we prove does not just resolve the Dyn-Farkhi conjecture in $\mathbb{R}^2$, but it provides the best upper bound for $d(A+B)$ in terms of $d(A)$ and $d(B)$. Such an optimal upper bound has not been found when $n\geq3$. For the case $n=1$ the optimal bound was found in \cite{meyer1}. If $A_1,\dots, A_m$ are compact in $\mathbb{R}$ such that $d(A_1)\geq \dots \geq d(A_m)$, the optimal bound in terms of $d(A_1),\dots, d(A_m)$ is
\begin{equation*}
    d\left(\sum_{i=1}^{m}A_i\right)\leq\max_{1\leq j\leq m}\left(d(A_j)-\sum_{i=j+1}^{m}d(A_i)\right).
\end{equation*}
An important part of the proof of Theorem \ref{dyn_farkhi_2_dimension} is a certain explicit characterization (observed in Lemma \ref{sumset_convexhull_two_dimension}) for $\textup{conv}(A+B)$ when $A$ and $B$ are compact sets in $\mathbb{R}^2$ with $\textup{dim}(A)=\textup{dim}(B)=2$. The lemma shows that the convex hull of a sumset has a very nice representation (which depends on $d(A)$ and $d(B)$) in $\mathbb{R}^2$. We are not aware if such a representation has been observed before.

A number of the papers already mentioned study functions that are in some sense measures of non-convexity. The survey \cite{fradelizi1} covers many of them. Another measure not mentioned above is the Schneider non-convexity index, which can be referenced in \cite{schneider1}.

The outline of the paper is as follows. In Section \ref{notation_examples} we will provide some basic notation and definitions. In Section \ref{proof} we will construct the proof of Theorem \ref{dyn_farkhi_2_dimension}. We conclude with some natural questions for further research in Section \ref{open_questions}.

\subsection{Acknowledgements} I thank Rober Fraser and Buma Fridman for helpful discussions. I also thank Robert Fraser for carefully reading this manuscript and suggesting corrections to improve its quality. 

\subsection{Funding} The author of this paper was supported in part by the National Science Foundation LEAPS - Division of Mathematical Sciences Grant No. 2316659.

\section{Notation and examples}\label{notation_examples}
The following notation will be used throughout the proof.
\begin{enumerate}
    \item \textbf{Dimension.} Let $A$ be a compact set in $\mathbb{R}^n$. The \textup{affine hull} of $A$, denoted $\textup{aff}(A)$, is the smallest (vector space dimension) translate of a subspace that contains $A$. The dimension of the set $A$, denoted by $\textup{dim}(A)$, is the dimension of $\textup{aff}(A)$. 
    \item \textbf{Boundary.} By a \textit{convex body} we mean a set $K\subseteq \mathbb{R}^n$ which is compact, convex, and has nonempty interior. The \textit{boundary} of a convex body $K$, denoted by $\partial K$, is defined to be the collection of all $x\in K$ such that every neighborhood of $x$ intersects both the interior of $K$ and the complement of $K$.
    \item \textbf{Triangles.} If $v_1,v_2,v_3$ are distinct points of $\mathbb{R}^2$, then we will denote the triangle with vertex set $V:=\{v_1,v_2,v_3\}$ by $\textup{conv}(V)$. To denote the side of the triangle $\textup{conv}(V)$ which has end points $v_i$ and $v_j$ we use the notation $v_iv_j$. The \textit{perpendicular bisector} of the side $v_iv_j$ is the unique line which passes through the midpoint $m_{ij}:=\frac{1}{2}(v_i+v_j)$ and is perpendicular to the side $v_iv_j$. To denote the perpendicular bisector, we will either use the notation $B_{ij}$ or $B(v_iv_j)$, whichever is more convenient. By the angle at $v_j$, we just mean the angle made by the sides $v_jv_k$ and $v_jv_i$ which meet at $v_j$. The same conventions will be used for parallelograms.
    \item \textbf{Hyperplanes.} In $\mathbb{R}^2$ a \textit{hyperplane} is the same as a line (or the translate of an $n-1$ dimensional subspace in $\mathbb{R}^n$). If $K\subset\mathbb{R}^2$ is a convex body, then a supporting hyperplane of $K$ is a line $H$ such that $H\cap K\neq\varnothing$, and $K$ is contained entirely in one of the two closed half spaces bounded by $H$. Every point in $\partial K$ has a supporting hyperplane. Note that in two dimensions the set $H\cap K$ is either a point or an interval. An \textit{extreme point} of $K$ is a point $x\in K$ for which there do not exist distinct $a,b\in K$ and $\lambda\in(0,1)$ such that $x= \lambda a+(1-\lambda)b$. The set of extreme points of $K$ is denoted by $\textup{extreme}(K)$. By the Krein-Milman theorem if $B$ is compact in $\mathbb{R}^n$, then $\textup{extreme}(\textup{conv}(B))\subseteq B$. It follows that if $H$ is any supporting hyperplane of $\textup{conv}(B)$, then $H\cap B\neq\varnothing$ (since $H$ contains extreme points of $\textup{conv}(B)$).
    \item \textbf{Distance.} To denote distances, we will use the notation $d(x,y):=|x-y|$, and for a set $A$ we will use the notation $d(x,A):=\inf_{a\in A}d(x,a)$. With this new notation the definition of Hausdorff distance to convex hull can be written as
    \begin{equation*}
        d(A)=\sup_{x\in\textup{conv}(A)}d(x,A).
    \end{equation*}
\end{enumerate}

In addition if we have a function of $\gamma$, then we will use the notation $D_{\gamma}$ to denote the derivative with respect to $\gamma$. If the real numbers $a$ and $b$ have the same sign, then we will say that $a\approx b$.

The next example shows that the inequality in Theorem \ref{dyn_farkhi_2_dimension} is optimal.

\begin{example}
    Let $d_A,d_B\geq0$. Set $A=\{0,2d_A\}\times\{0\}$ and $B=\{0\}\times\{0,2d_B\}$. Then $d(A)=d_A$ and $d(B)=d_B$. The sumset is $A+B=\{0,2d_A\}\times\{0,2d_B\}$. The Hausdorff distance to convex hull is achieved in the center of the rectangle $A+B$. This is $d(A+B)=\sqrt{d_A^2+d_B^2}$, which verifies the equality part of Theorem \ref{dyn_farkhi_2_dimension}.
\end{example}

Recall that from the introduction that under certain conditions we can have $v(A)=d(A)$, and also the relation $d(A)\leq \textup{rad}(A)$. The next example will show that neither of these is always equality, so it is not possible to use Theorem \ref{cassels_theorem} or \ref{starr_theorem} to prove subadditivity of $d^2$.

\begin{example}
    Let $A=\{(-2,0),(2,0),(0,1)\}$ be the vertex set of an obtuse triangle. We first compute a lower bound on $v(A)$. Set $x=(0,0)\in\textup{conv}(A)$. Then we can write $x=\frac{1}{2}(-2,0)+\frac{1}{2}(2,0)$, and this representation of $x$ is unique. Therefore, 
    \begin{equation*}
        v^2(A)\geq \frac{1}{2}(2)^2+\frac{1}{2}(2)^2=4.
    \end{equation*}
    The lower bound is $v(A)\geq 2$. To compute $d(A)$ we rely on Lemma \ref{hausdorff_distance_obtuse_triangle}. Let $\theta$ be the angle at $(2,0)$. The smaller side lengths are both $\sqrt{5}$. Then the Hausdorff distance to convex hull is
    \begin{equation*}
        d(A)=\frac{\sqrt{5}}{2\cos(\theta)}=\frac{5}{4}.
    \end{equation*}
    We conclude that $v(A)>d(A)$. We can also see that $\textup{rad}(A)$ is the radius of a disk which contains $\textup{conv}(A)$, but a circle with radius $d(A)$ cannot contain $\textup{conv}(A)$. Then $\textup{rad}(A)>d(A)$.
\end{example}

\section{Proof of Theorem \ref{dyn_farkhi_2_dimension}}\label{proof}

The proof of Theorem \ref{dyn_farkhi_2_dimension} is structured as follows.
\begin{enumerate}
    \item Over all the parallelograms with fixed side lengths $a$ and $x$, the largest possible value of $d(\textup{vert}(P))$ is obtained when $P$ is a rectangle. This is done in Lemmas \ref{hausdorff_distance_acute_triangle}, \ref{hausdorff_distance_obtuse_triangle}, \ref{hausdorff_distance_parallelogram}.
    \item If $x$ belongs to $\textup{conv}(A)+B$, then $d(x,A+B)\leq d(A)$. This is Lemma \ref{one_set_convex_lemma}.
    \item Suppose that $A$ and $B$ are compact in $\mathbb{R}^n$. We show that if $\textup{dim}(A)=1$, then Theorem \ref{dyn_farkhi_2_dimension} is true (note that the dimension here can be larger than $2$). This is Lemma \ref{one_set_one_dimensional}.
    \item We give a characterization of the sumset $\textup{conv}(A+B)$ of the convex hull of full dimensional compact sets in $\mathbb{R}^2$. It turns out that in this case an element of $\textup{conv}(A+B)$ either belongs to $\textup{conv}(A)$ translated by an element of $B$, or $\textup{conv}(B)$ translated by an element of $A$, or a parallelogram with vertex set in $A+B$ with side lengths at most $2d(A)$ and $2d(B)$ respectively. Using the earlier result about parallelograms finishes the proof. We establish this in Lemma \ref{sumset_convexhull_two_dimension}.
\end{enumerate}

\begin{lemma}\label{hausdorff_distance_acute_triangle}
    Let $T$ be an acute triangle with angle-side opposite pairs $(\alpha,a)$, $(\beta,b)$, $(\gamma,c)$. Then
    \begin{equation*}\\
        \begin{split}
             d(\textup{vert}(T))&=\frac{c}{2\sin(\gamma)}\\
             &=\frac{a}{2\sin(\alpha)}\\
             &=\frac{b}{2\sin(\beta)}.
        \end{split}
    \end{equation*}
\end{lemma}

\begin{proof}
   Let $v_1,v_2,v_3$ be the vertices of $T$ which correspond to the angles $\alpha,\beta,\gamma$ respectively. Sine $T$ is acute, the bisectors $B_{12}$, $B_{13}$, $B_{23}$ intersect at a point $P\in\textup{int}(T)$. Dissect $T$ into the triangles $\textup{conv}\{v_1,P,v_2\}$, $\textup{conv}\{v_1,P,v_3\}$, $\textup{conv}\{v_2,P,v_3\}$, which are disjoint except possibly at their boundaries. Let $m_{ij}$ denote the midpoint of the side $v_iv_j$. By considering the midpoints $m_{ij}$, each of the above triangles can be dissected into two right triangles. For example, we can write
   \begin{equation*}
       \textup{conv}\{v_1,P,v_2\}=\textup{conv}\{v_1,m_{12},P\}\cup\textup{conv}\{v_2,m_{12},P\}=:T_1\cup T_2.
   \end{equation*}
   If $x\in T_1$, then since $v_1$ is separated from $v_2$ and $v_3$ by the bisectors $B(v_1v_2)$ and $B(v_1v_3)$, we have $d(x,\textup{vert}(T))=d(x,v_1)$. The farthest distance from $v_1$ within the triangle $T_1$ is the opposite end of the hypotenuse (i.e. when $x=P$), so we have $d(x,v_1)\leq d(P,v_1)$. Equality holds if and only if $x=P$. Similarly, if $x\in T_2$, then $d(x,\textup{vert}(T))=d(x,v_2)\leq d(P,v_2)$, where equality holds if and only if $x=P$. But, also $d(P,v_1)=d(P,v_2)=d(P,v_3)$, so we have $d(x,\textup{vert}(T))\leq d(P,v_1)$ for $x\in T_1\cup T_2$. Now repeat the exact same process for the remaining triangles in the above dissection of $T$. Then we have for any $x\in T$, $d(x,\textup{vert}(T))\leq d(P,v_1)$, where equality is possible. Therefore, $d(\textup{vert}(T))=d(P,v_1)$. Now we can compute the formula for $d(P,v_1)$. First recall that the vertex $v_1$ corresponds to the angle $\alpha$. We consider the triangles $A:=\textup{conv}\{v_1,m_{12},P\}$ and $B:=\textup{conv}\{v_1,m_{13},P\}$, which have a common hypotenuse of length $h:=d(P,v_1)$. If triangle $A$ has an angle $\theta$ at vertex $v_1$, then the triangle $B$ has an angle $\alpha-\theta$ at vertex $v_1$. using basic trigonometry we can solve for $h$ in terms of $\theta$ and $\alpha-\theta$:
   \begin{equation*}
       \frac{c}{2\cos(\theta)}=h=\frac{b}{2\cos(\alpha-\theta)}.
   \end{equation*}
   Using the above identity and the identity $\cos(\alpha-\theta)=\cos(\alpha)\cos(\theta)+\sin(\alpha)\sin(\theta)$ we can solve for $\sin(\theta)$:
   \begin{equation*}
       \sin(\theta)=\frac{(b-c\cos(\alpha))}{c\sin(\alpha)}\cos(\theta).
   \end{equation*}
   Now using the Pythagorean identity we have
   \begin{equation*}
       \cos(\theta)=\frac{c\sin(\alpha)}{\sqrt{b^2+c^2-2bc\cos(\alpha)}}.
   \end{equation*}
   Substitute the above expression for $\cos(\theta)$ into the formula for $h$ to get
   \begin{equation*}
       h=\frac{\sqrt{b^2+c^2-2bc\cos(\alpha)}}{2\sin(\alpha)}=\frac{a}{2\sin(\alpha)}.
   \end{equation*}
   This verifies one of the formulae. By symmetry the other two formulae must also be true, so the proof is complete.
\end{proof}

\begin{lemma}\label{hausdorff_distance_obtuse_triangle}
    Let $T$ be an obtuse triangle with angle-side opposite pairs $(\alpha,a)$, $(\beta,b)$, $(\gamma,c)$. Assume that the side lengths satisfy $a\leq b<c$. Then
    \begin{equation*}
        d(\textup{vert}(T))=\frac{b}{2\cos(\alpha)}.
    \end{equation*}
\end{lemma}

\begin{proof}
    Denote the vertices of $T$ by $v_1,v_2,v_3$, where the side of length $a$ is $v_1v_2$, and the side of length $b$ is $v_2v_3$. For given $i\neq j$, let $B_{ij}$ denote the perpendicular bisector of the side $v_iv_j$. The perpendicular bisectors $B(v_1v_2)$, $B(v_2v_3)$ intersect the side $v_1v_3$ at the points $P_{12}$,$P_{23}$ respectively. Let $m_{ij}$ denote the midpoint of the side $v_iv_j$. We now dissect the triangle $T$ into three triangles which are disjoint except possibly at the boundaries:
    \begin{equation*}
        T=\textup{conv}\{v_1,v_2,P_{12}\}\cup\textup{conv}\{v_2,P_{23},P_{12}\}\cup\textup{conv}\{v_2,v_3,P_{23}\}.
    \end{equation*}
    The triangle $\textup{conv}\{v_1,v_2,P_{12}\}$ is the union of two right triangles. Suppose for example that $x$ belongs to the right triangle with vertex $v_1$. Since $v_1$ is separated from $v_2$ and $v_3$ by the bisectors $B_{12}$ and $B_{13}$ respectively, we have $d(x,\textup{vert}(T))=d(x,v_1)$. Since $v_1$ and $P_{12}$ are the endpoints of the hypotenuse, we have $d(x,v_1)\leq d(P_{12},v_1)$. That is, $d(x,\textup{vert}(T))\leq d(P_{12},v_1)$. If $x$ belongs to the other right triangle with vertex $v_2$, then by similar reasoning we find that $d(x,\textup{vert}(T))\leq d(P_{12},v_2)$. Since $d(P_{12},v_1)=d(P_{12},v_2)$, we have shown that when $x$ is in the triangle $\textup{conv}\{v_1,v_2,P_{12}\}$, we have $d(x,\textup{vert}(T))\leq d(P_{12},v_2)$. Equality holds if and only if $x=P_{12}$. To calculate the distance $d(P_{12},v_2)$, we first observe that the right triangle has angle $\beta$ at vertex $P_{12}$ with adjacent side $a/2$, and we want to compute the hypotenuse. With basic trigonometry we compute 
    \begin{equation*}
        d(P_{12},v_2)=\frac{a}{2\cos(\beta)}.
    \end{equation*}
    Just as was observed above, the triangle $\textup{conv}\{v_2,v_3,P_{23}\}$ is the union of two right triangles, where the hypotenuses meet at the point $P_{23}$, and one of the triangles has angle $\alpha$ at vertex $v_2$. Just as was shown above we find that if $x$ belongs to either right triangle, then $d(x,\textup{vert}(T))\leq d(P_{23},v_2)$. By a similar computation we compute
    \begin{equation*}
        d(P_{23},v_2)=\frac{b}{2\cos(\alpha)}.
    \end{equation*}
    Finally, the triangle $\textup{conv}\{v_2,P_{12},P_{23}\}$ has vertex $v_2$ which is separated from the vertices $v_1$,$v_3$ by the bisectors $B(v_1v_2)$, $B(v_2v_3)$ respectively. So, in the triangle we have $d(x,\textup{vert}(T))=d(x,v_2)$. The farthest distance from the vertex $v_2$ is the endpoint of the longest of the two sides which have $v_2$ as an endpoint. These lengths are $d(P_{12},v_2)$ and $d(P_{23},v_2)$. Therefore, using the above computations we have for such an $x$,
    \begin{equation*}
        d(x,\textup{vert}(T))\leq\max\left\{\frac{a}{2\cos(\beta)},\frac{b}{2\cos(\alpha)}\right\}.
    \end{equation*}
    Equality holds when $x=P_{12}$ or $x=P_{23}$, whichever achieves the maximum. So, we have proved that for $x\in T$,
    \begin{equation*}
        d(\textup{vert}(T))=\max\left\{\frac{a}{2\cos(\beta)},\frac{b}{2\cos(\alpha)}\right\}.
    \end{equation*}
    All that remains is to determine which of the two distances is the largest. Use the law of cosines to get
    \begin{equation*}
        \begin{split}
            \frac{b}{2\cos(\alpha)}-\frac{a}{2\cos(\beta)}&=\frac{b^2c}{b^2+c^2-a^2}-\frac{a^2c}{a^2+c^2-b^2}\\
            &\approx b^2(a^2+c^2-b^2)-a^2(b^2+c^2-a^2)\\
            &=(b^2-a^2)(c^2-b^2-a^2)\geq0.
        \end{split}
    \end{equation*}
    The last step followed from the law of cosines (with angle $\gamma$, which is larger than $90^{\circ}$) and that $b\geq a$. Putting everything together we have proved
    \begin{equation*}
        d(\textup{vert}(T))=\frac{b}{2\cos(\alpha)}.
    \end{equation*}
    This completes the proof.
\end{proof}

\begin{remark}
    Note that in Lemmas \ref{hausdorff_distance_acute_triangle} and \ref{hausdorff_distance_obtuse_triangle} we did not consider the case where $T$ is a right triangle. In either Lemma, if we carefully follow the proof, then it can be seen that for a right triangle we have $d(\textup{vert}(T))=\frac{c}{2}$, where $c$ is the length of the hypotenuse. 
\end{remark}

\begin{lemma}\label{hausdorff_distance_parallelogram}
    Let $0<a\leq x$, and let $P_{\gamma}$ be the parallelogram with side lengths $a$ and $x$, where $\gamma$ is the smaller angle which satisfies $0\leq \cos(\gamma)<1$. Then $d(\textup{vert}(P_{\gamma}))\leq d(\textup{vert}(P_{90^{\circ}}))$.
\end{lemma}

\begin{proof}
    Label the vertices by $v_1,v_2,w_1,w_2$ so that $v_1w_1$, $v_2w_2$ are the sides of length $a$, and $v_1v_2$, $w_1w_2$ are the sides of length $x$. We consider the following cases.
    \begin{case}
        $0< \cos(\gamma)\leq \frac{a}{x}$.
    \end{case}
    Define the triangle $T:=\textup{conv}\{v_1,v_2,w_1\}$. We will verify that $T$ is not an obtuse triangle. Let the angle-side-opposite pairs be given by $(\theta,x)$, $(\alpha,a)$, and $(\gamma,h)$. By the law of cosines and the assumption we have
    \begin{equation*}
        h^2=a^2+x^2-2ax\cos(\gamma)\geq a^2+x^2-2a^2=x^2-a^2.
    \end{equation*}
    Again by the law of cosines we have
    \begin{equation*}
        \cos(\theta)=\frac{a^2+h^2-x^2}{2ah}\approx a^2+h^2-x^2\geq a^2+x^2-a^2-x^2=0.
    \end{equation*}
    Then $\cos(\theta)\geq0$, which implies that $\theta\leq 90^{\circ}$. Using law of cosines once more we have
    \begin{equation*}
        \cos(\alpha)=\frac{x^2+h^2-a^2}{2xh}\approx x^2+h^2-a^2\geq x^2+x^2-a^2-a^2=2x^2-2a^2\geq0.
    \end{equation*}
    Then $\cos(\alpha)\geq0$, which implies that $\alpha\leq 90^{\circ}$. Since none of the angles can be larger than $90^{\circ}$, we have verified that the triangle $T$ is not obtuse. Then $d(\textup{vert}(T))$ is achieved at a point $x_H\in T$ which lies at the intersection of all three perpendicular bisectors. In other words, 
    \begin{equation*}
        d(\textup{vert}(T))=d(x_H,v_1)=d(x_H,v_2)=d(x_H,w_1).
    \end{equation*}
    The bisector $B(w_1w_2)$ separates $P_\gamma$ into two regions, and one of those regions must contain the vertex $w_2$. Since $x_H\in B(v_1v_2)$, we reason that $x_H$ belongs to the region that does not contain $w_2$. Therefore $d(x_H,w_1)<d(x_H,w_2)$. Now consider the triangle $T^*:=\textup{conv}\{w_1,w_2,v_2\}$. By the exact same reasoning ($T^*$ is in fact the same triangle as $T$ but reflected across the line through $w_1$ and $v_2$) we find a point $x_H^*\in T^*$ which satisfies 
    \begin{equation*}
        d(\textup{vert}(T*))=d(x_H^*,w_1)=d(x_H^*,w_2)=d(x_H^*,v_2),
    \end{equation*}
    and $d(x_H^*,v_2)<d(x_H^*,v_1)$. Now, to calculate the Hausdorff distance of $\textup{vert}(P_\gamma)$ we first note that $P_\gamma=T\cup T^*$. Let $x\in P_\gamma$. Then $x\in T$ or $x\in T^*$. Without loss of generality, we may assume that $x\in T$. Then 
    \begin{equation*}
        \begin{split}
            \min\{d(x,v_1),d(x,v_2),d(x,w_1),d(x,w_2)\}&\leq \min\{d(x,v_1),d(x,v_2),d(x,w_1)\}\\
            &\leq \sup_{t\in T}\min\{d(t,v_1),d(t,v_2),d(t,w_1)\}\\
            &=d(\textup{vert}(T)).
        \end{split}
    \end{equation*}
    The same inequality occurs if $x\in T^*$. This implies that $d(\textup{vert}(P_\gamma))\leq d(\textup{vert}(T))$. But we have shown that the element $x_H$ satisfies 
    \begin{equation*}
        \min\{d(x_H,v_1),d(x_H,v_2),d(x_H,w_1),d(x_H,w_2)\}=\min\{d(x_H,v_1),d(x_H,v_2),d(x_H,w_1)\}=d(\textup{vert}(T)).
    \end{equation*}
    Therefore, $d(\textup{vert}(P_\gamma))=d(\textup{vert}(T))$. By Lemma \ref{hausdorff_distance_acute_triangle} we have that
    \begin{equation*}
        d(\textup{vert}(T))=\frac{\sqrt{a^2+x^2-2ax\cos(\gamma)}}{2\sin(\gamma)}.
    \end{equation*}
    Now, setting $d_\gamma:=d(\textup{vert}(T))$ we compute
    \begin{equation*}
        \begin{split}
            D_\gamma[(2d_\gamma)^2]&=D_\gamma\left[\frac{a^2+x^2-2ax\cos(\gamma)}{\sin^2(\gamma)}\right]\\
            &\approx ax \sin^2(\gamma)-\cos(\gamma)(a^2+x^2)+2ax\cos^2(\gamma)\\
            &=ax\cos^2(\gamma)-(a^2+x^2)\cos(\gamma)+ax=:P(\cos(\gamma)).
        \end{split}
    \end{equation*}
    The zeros of the polynomial $P(\cos(\gamma))$ are $\frac{a}{x}$ and $\frac{x}{a}$. Since the polynomial $P(\cos(\gamma))$ is (in the variable $\cos(\gamma)$) a parabola opening upwards, and since $\cos(\gamma)\leq \frac{a}{x}<\frac{x}{a}$ by assumption, we must have $P(\cos(\gamma))\geq0$. So, $D_\gamma[(2d_\gamma)^2]\geq0$, from which it follows that $D_\gamma[d_\gamma]\geq0$. This proves the result for the given case.

    \begin{case}
        $\frac{a}{x}<\cos(\gamma)<1$.
    \end{case}
    We continue to use the same notation as in the first case. In this case we will verify that the triangle $T$ (and so also $T^*$) is obtuse. To verify this first note that by using the computations from the first case (but with a flipped inequality sign) we find that $h^2<x^2-a^2$. Then we compute
    \begin{equation*}
        \cos(\theta)\approx a^2+h^2-x^2<a^2+x^2-a^2-x^2=0.
    \end{equation*}
    Then $\cos(\theta)<0$, which implies that $\theta>90^{\circ}$. So, the triangle $T$ must be obtuse where the longest side has length $x$. Then the two possible points where $d(\textup{vert}(T))$ is achieved are given by 
    \begin{equation*}
        \begin{split}
            \{x_1\}&=B(v_1w_1)\cap v_1v_2,\\
            \{x_2\}&=B(w_1v_2)\cap v_1v_2.
        \end{split}
    \end{equation*}
    Since the vertices $v_1$ and $v_2$ are separated from $w_2$ by the bisector $B(v_2w_2)$ we find that $d(\textup{vert}(P_\gamma))=d(\textup{vert}(T))$ by the same reasoning as in the first case. Using Lemma \ref{hausdorff_distance_obtuse_triangle} we find that 
    \begin{equation*}
        d(\textup{vert}(T))=\max\left(\frac{a}{2\cos(\gamma)},\frac{a^2+x^2-2ax\cos(\gamma)}{2x-2a\cos(\gamma)}\right).
    \end{equation*}
    Note that to compute the second component we used the formula $h/(2\cos(\alpha))$ and law of cosines for both $h$ and $\cos(\alpha)$. We compute the derivative of the first component to get 
    \begin{equation*}
        D_\gamma\left[\frac{a}{2\cos(\gamma)}\right]\approx 2a\sin(\gamma)\geq0.
    \end{equation*}
    Then we compute the derivative of the second component to get
    \begin{equation*}
        D_\gamma\left[\frac{a^2+x^2-2ax\cos(\gamma)}{2x-2a\cos(\gamma)}\right]\approx x^2-a^2\geq0.
    \end{equation*}
    Since both components are increasing functions, the maximum of the two components is an increasing function, so we have proved the claim.
\end{proof}

\begin{lemma}\label{one_set_convex_lemma}
    Let $A$ and $B$ be compact sets in $\mathbb{R}^n$. Suppose that $x\in \textup{conv}(A)+B$. Then
    \begin{equation*}
        d(x,A+B)\leq d(A).
    \end{equation*}
\end{lemma}

\begin{proof}
    If $x\in\textup{conv}(A)+B$, then we can write $x=a+b$, where $a\in\textup{conv}(A)$ and $b\in B$. Now, observe that
    \begin{equation*}
        \begin{split}
            d(a+b,A+B)&\leq d(a+b,A+b)\\
            &=\inf_{a^{\prime}+b\in A+b}|(a+b)-(a^{\prime}+b)|\\
            &=\inf_{a^{\prime}\in A}|a-a^{\prime}|\\
            &=d(a,A)\\
            &\leq d(A).
        \end{split}
    \end{equation*}
    So, we have $d(x,A+B)\leq d(A)$ for any $x\in \textup{conv}(A)+B$, which completes the proof.
\end{proof}

\begin{remark}
    By Lemma \ref{one_set_convex_lemma} we have the following natural result: If $A$ and $B$ are compact sets in $\mathbb{R}^n$ such that $B$ is convex, then $d(A+B)\leq d(A)$. 
\end{remark}

\begin{lemma}\label{one_set_one_dimensional}
    For an integer $n\geq2$, let $A$ and $B$ be compact sets in $\mathbb{R}^n$ such that $\textup{dim}(A)=1$ and $\textup{dim}(B)\in\{1,\dots,n\}$. Then
    \begin{equation*}
        d^2(A+B)\leq d^2(A)+d^2(B).
    \end{equation*}
\end{lemma}

\begin{proof}
    Let $a\in\textup{conv}(A)$ and let $b\in \textup{conv}(B)$. Since $\textup{dim}(A)=1$ there exists an interval $I_a\subseteq \textup{conv}(A)$ such that $\textup{vert}(I_a)\subseteq A$, $a\in I_a$, and $d(\textup{vert}(I_a))\leq d(A)$. Without loss of generality we assume that $I_a=[0,x]$ (if not, just translate $A$ appropriately to achieve this representation). By compactness of $B$ there exists $\gamma\in B$ such that $d(b,B)=d(b,\gamma)=:d_\gamma$. By definition of Hausdorff distance, $d_{\gamma}\leq d(B)$. Also, $a+b\in [b,b+x]$.
    We will now measure the distance from $a+b$ to $\gamma$. We begin by constructing two triangles. The first triangle, call it $T_1$, has vertices $b$, $\gamma$, $b+d_a(\theta)u$, where $u$ is a unit vector pointing in the direction of $[b,b+x]$, and
    \begin{equation*}
        d_a(\theta):=d_\gamma\cos(\theta)+\sqrt{d^2(A)+d^2(B)-d_\gamma^2\sin^2(\theta)}, 
    \end{equation*}
    and $\theta$ is the angle at the vertex $b$. Also, we will be assuming that $\theta\in(0,\frac{\pi}{2}]$. In the case that $\theta$ is obtuse we can use a similar idea, and we will deal with the case $\theta=0$ later. The second triangle, call it $T_2$, has vertices $b+d_a(\theta)u$, $\gamma+x$, and $b+x$. Now, either $a+b\in [b,b+d_a(\theta)u]$ or $a+b\in [b+d_a(\theta)u,b+x]$. In the first case we have
    \begin{equation*}
        \begin{split}
            d^2(a+b,\gamma)&\leq \max(d_\gamma^2,d^2(b+d_a(\theta)u,\gamma))\\
            &=\max(d_\gamma^2,d_a(\theta)^2+d_\gamma^2-2d_a(\theta)d_\gamma\cos(\theta))\\
            &\leq d^2(A)+d^2(B).
        \end{split}
    \end{equation*}
    Suppose now that the second case holds. The angle at the vertex $b+x$ is given by $\pi-\theta$. The side lengths are $|I_a|-d_a(\theta)$, $d_\gamma$, and $d(b+d_a(\theta)u,\gamma+x)$. We have 
    \begin{equation*}
        \begin{split}
            d^2(a+b,\gamma+x)&\leq d^2(b+d_a(\theta)u,\gamma+x)\\
            &=(|I_a|-d_a(\theta))^2+d_\gamma^2-2(|I_a|-d_a(\theta))d_{\gamma}\cos(\pi-\theta)\\
            &=|I_a|^2-2|I_a|\sqrt{d^2(A)+d^2(B)-d_\gamma^2\sin^2(\theta)}+d^2(A)+d^2(B)\\
            &\leq |I_a|(|I_a|-2d(A))+d^2(A)+d^2(B)\\
            &\leq d^2(A)+d^2(B).
        \end{split}
    \end{equation*}
    In the second to last step we used that $|I_a|\leq 2d(A)$. 
    
    Finally, we consider with the case where $[b,\gamma]$ and $[0,x]$ are parallel to each other. This is equivalent to the one-dimensional problem of adding the sets $\{0,x\}$ and $\{b,\gamma\}$. Assume that $x\geq 0$. The case where $x\leq 0$ is similar. If $b\leq \gamma$, then
    \begin{equation*}
        \{0,x\}+\{b,\gamma\}=\{b,\gamma,b+x,\gamma+x\}\supseteq\{b,\gamma,x+\gamma\}.
    \end{equation*}
    Then $a+b$ belongs to $[b,\gamma]$ or $[\gamma,x+\gamma]$. In the first case, use that $\gamma\in A+B$ and that $\gamma-b\leq d(B)$ to conclude that $d(a+b,A+B)\leq d(B)$. In the second case use that $\gamma,x+\gamma\in A+B$ and $(x+\gamma)-\gamma=x\leq 2d(A)$ to conclude that $d(a+b,A+B)\leq d(A)$. If $\gamma\leq b$, then 
    \begin{equation*}
        \{0,x\}+\{\gamma,b\}=\{\gamma,b,\gamma+x,b+x\}\supseteq\{\gamma,\gamma+x,b+x\}.
    \end{equation*}
    Then $a+b$ belongs to either $[\gamma,\gamma+x]$ or $[\gamma+x,b+x]$. In the first case, since $\gamma,\gamma+x\in A+B$ and $(\gamma+x)-\gamma=x\leq 2d(A)$ we conclude that $d(a+b,A+B)\leq d(A)$. In the second case, since $\gamma+x\in A+b$ and $(b+x)-(\gamma+x)=b-\gamma\leq d(B)$, we conclude that $d(a+b,A+B)\leq d(B)$. 
    
    Therefore, in any case we have shown that if $a+b\in\textup{conv}(A+B)$, then $d^2(a+b,A+B)\leq d^2(A)+d^2(B)$, which proves the lemma.
\end{proof}

\begin{lemma}\label{sumset_convexhull_two_dimension}
    Let $A$ and $B$ be compact sets in $\mathbb{R}^2$ which contain $0$ and such that $\textup{dim}(A)=\textup{dim}(B)=2$. Suppose that 
    \begin{equation}\label{convex_sumset_without_translates}
        x\in \textup{conv}(A+B)\cap[A+\textup{conv}(B)]^c\cap[\textup{conv}(A)+B]^c.
    \end{equation}
    Then there exist $a_1,a_2\in A$ and $b_1,b_2\in B$ such that $d(a_1,a_2)\leq 2d(A)$, $d(b_1,b_2)\leq 2d(B)$, and $x\in [a_1,a_2]+[b_1,b_2]$. That is, $x$ belongs to a parallelogram with side lengths at most $2d(A)$ and $2d(B)$ respectively, and vertex set in the sumset $A+B$.
\end{lemma}

\begin{proof}
    We will verify that 
    \begin{equation}\label{convex_sumset_bdry_included}
        \textup{conv}(A+B)=\textup{conv}(A)\cup (\partial \textup{conv}(A)+\textup{conv}(B)).
    \end{equation}
    Since $0\in\textup{conv}(B)$, it is immediate that the right side is contained in the left side. Let $x\in \textup{conv}(A)+\textup{conv}(B)$. Then $x=a+b$, where $a\in\textup{conv}(A)$, $b\in\textup{conv}(B)$. If $x\in\textup{conv}(A)$, then we are done. So, assume that $x\notin \textup{conv}(A)$. Then, since $a\in\textup{conv}(A)$ and $a+b\notin \textup{conv}(A)$, we have $[a,a+b]\cap \partial \textup{conv}(A)\neq\varnothing$. Choose $z\in [a,a+b]\cap \partial \textup{conv}(A)$. Then there exists $\lambda\in[0,1]$ such that $z=\lambda a+(1-\lambda)(a+b)$. This is equivalent to $a+b=(1-\lambda)z+\lambda (b+z)$. That is, $a+b\in z+[0,b]\subseteq \partial \textup{conv}(A)+\textup{conv}(B)$. This verifies (\ref{convex_sumset_bdry_included}). If $\partial\textup{conv}(A)\subseteq A$, the set in (\ref{convex_sumset_without_translates}) is empty by (\ref{convex_sumset_bdry_included}). So we will assume that $\partial\textup{conv}(A)\not\subseteq A$. Using the equation 
    \begin{equation*}
        \partial\textup{conv}(A)=(\partial\textup{conv}(A)\cap A)\cup (\partial\textup{conv}(A)\cap A^c),
    \end{equation*}
    we can write (\ref{convex_sumset_bdry_included}) as
    \begin{equation*}
        \textup{conv}(A+B)=\textup{conv}(A)\cup[\textup{conv}(B)+\partial\textup{conv}(A)\cap A]\cup[\textup{conv}(B)+\partial \textup{conv}(A)\cap A^c].
    \end{equation*}
    Let $x\in \textup{conv}(A+B)$. Then by assumption, $x\notin[\textup{conv}(B)+\partial\textup{conv}(A)\cap A]$, and $x\notin\textup{conv}(A)$. The only choice then is $x\in \textup{conv}(B)+\partial\textup{conv}(A)\cap A^c$. Write $x=a+b$, where $a\in\partial\textup{conv}(A)\cap A^c$, and $b\in\textup{conv}(B)$. We can find $\gamma_1,\gamma_2\in A$ which are different from each other such that $a\in [\gamma_1,\gamma_2]$. Then $x=a+b\in \textup{conv}(B)+[\gamma_1,\gamma_2]$. Let $L_x$ be the line which passes through $x$ and is parallel to the interval $[\gamma_1,\gamma_2]$. Then $L_x$ intersects the sets $\textup{conv}(B)+\gamma_1$ and $\textup{conv}(B)+\gamma_2$ at translates of some interval $I_x\subseteq\textup{conv}(B)$. We have $x\in I_x+[\gamma_1,\gamma_2]$. Now, let $L_x^{+}$ and $L_x^{-}$ represent the closed half spaces bounded by $L_x$. We will show that there exist $b_1,b_2\in B$ such that $b_1+\gamma_1\in L_x^+$, $b_2+\gamma_1\in L_x^-$, and $d(b_1,b_2)\leq 2d(B)$. We make the following claim:
    \begin{claim}
        There exist $b_1,b_2\in B$ which satisfy
        \begin{equation*}
            d(b_1,b_2)=\min_{\substack{b_1^{\prime}+\gamma_1\in L_x^+\cap(B+\gamma_1)\\b_2^{\prime}+\gamma_1\in L_x^-\cap(B+\gamma_1)}}d(b_1^{\prime},b_2^{\prime}).
        \end{equation*}
    \end{claim}
    \begin{proof}
        Define $R^+$ and $R^-$ by $R^{\pm}:=L_x^{\pm}\cap(B+\gamma_1)$. We must first show that $R^+$ and $R^-$ are nonempty. By the Krein-Milman theorem $\textup{extreme}(\textup{conv}(B))\subseteq B$. Suppose first that $L_x$ is a supporting hyperplane of $\textup{conv}(B)+\gamma_1$. Then the set $L_x\cap(\textup{conv}(B)+\gamma_1)$ contains extreme points in $\textup{conv}(B)+\gamma_1$, and so points in $B+\gamma_1$. So, the sets $R^{\pm}$ must be nonempty. Now suppose that $L_x$ is not a supporting hyperplane. Then the sets $S^{\pm}:=(\textup{conv}(B)+\gamma_1)\cap L_x^{\pm}$ are convex sets (contained in $\textup{conv}(B)+\gamma_1$), and therefore are the convex hulls of their extreme points. Since both $S^{\pm}$ contain interior points of $\textup{conv}(B)+\gamma_1$, each of them must have at least one extreme point from $\textup{conv}(B)+\gamma_1$, and therefore each set has points in $B+\gamma_1$. This verifies that the sets $R^{\pm}$ are nonempty. Denote the right side minimum by $\alpha$. Consider a sequence $d_k:=d(b_1^k,b_2^k)$ such that $d_k\rightarrow\alpha$, and $b_1^k+\gamma_1\in R^+$, $b_2^k+\gamma_1\in R^-$. By compactness, we can assume that (up to subsequences) $b_1^k\rightarrow b_1\in B$ and $b_2^k\rightarrow b_2\in B$. Then
        \begin{equation*}
            |b_1-b_2|\leq |b_1-b_1^k|+|b_2^k-b_2|+|(b_1^k-b_2^k)|\rightarrow \alpha.
        \end{equation*}
        That is, $|b_1-b_2|\leq\alpha$. But, $b_1+\gamma_1\in R^+$ and $b_2+\gamma_1\in R^-$. So, we must have $|b_1-b_2|\geq\alpha$. This proves $|b_1-b_2|=\alpha$.
    \end{proof}
    Now, with the existence of $b_1,b_2$ established, we will show that $d(b_1,b_2)\leq 2d(B)$. Suppose for contradiction that $d(b_1,b_2)>2d(B)$. Then $b^*:=\frac{1}{2}(b_1+b_2)\in\textup{conv}(B)$, and the ball $B(b^*,d(b^*,b_1))$ does not have any interior points which belong to $B$ (or else we contradict the minimality of $d(b_1,b_2)$). But then $d(b^*,B)=d(b^*,b_1)>d(B)$, which is a contradiction. So we must have $d(b_1,b_2)\leq 2d(B)$. It remains to show that $x\in[\gamma_1,\gamma_2]+[b_1,b_2]$. First, we note that $[b_1,b_2]\cap I_x=\{y\}$. We also know that $x\in I_x+[\gamma_1,\gamma_2]$. Since $x\notin A+\textup{conv}(B)$ we have $x\notin I_x+\gamma_1$ and $x\notin I_x+\gamma_2$. Then 
    \begin{equation*}
    x\in[y+\gamma_1,y+\gamma_2]=y+[\gamma_1,\gamma_2]\subseteq [b_1,b_2]+[\gamma_1,\gamma_2].
    \end{equation*}
    Now, $x\in [\gamma_1,\gamma_2]+[b_1,b_2]\subseteq \textup{conv}(A)+[b_1,b_2]$. By repeating the exact same argument as was shown above, we can find $a_1,a_2\in A$ such that $d(a_1,a_2)\leq 2d(A)$ and $x\in [a_1,a_2]+[b_1,b_2]$. This completes the proof.
\end{proof}

\begin{proof}[Proof of Theorem \ref{dyn_farkhi_2_dimension}]
Let $A$ and $B$ be compact in $\mathbb{R}^2$. If $\textup{dim}(A)=0$, then using translation invariance we have $d^2(A+B)=d^2(B)=d^2(A)+d^2(B)$. The same holds if $\textup{dim}(B)=0$. Suppose now that $\textup{dim}(A)=1$ and that $\textup{dim}(B)\in\{1,2\}$. By Lemma \ref{one_set_one_dimensional} we have $d^2(A+B)\leq d^2(A)+d^2(B)$. The same holds if $\textup{dim}(B)=1$ and $\textup{dim}(A)\in\{1,2\}$. The only case we are left with is $\textup{dim}(A)=\textup{dim}(B)=2$. Let $x\in \textup{conv}(A+B)$. If $x\in [A+\textup{conv}(B)]\cup [\textup{conv}(A)+B]$, then we use Lemma \ref{one_set_convex_lemma} to obtain $d^2(x,A+B)\leq \max\{d^2(A),d^2(B)\}\leq d^2(A)+d^2(B)$. Otherwise, 
\begin{equation*}
    x\in \textup{conv}(A+B)\cap[A+\textup{conv}(B)]^c\cap[\textup{conv}(A)+B]^c.
\end{equation*}
By Lemma \ref{sumset_convexhull_two_dimension} $x$ must belong to a parallelogram with vertex set in $A+B$ and side lengths at most $2d(A)$ and $2d(B)$ respectively. By Lemma \ref{hausdorff_distance_parallelogram} the Hausdorff distance to convex hull of the vertex set of a parallelogram is maximized for rectangles. Therefore we must have (using that the Hausdorff distance to convex hull of the vertex set of a rectangle is obtained in its center), denoting the parallelogram containing $x$ by $P$,
\begin{equation*}
    \begin{split}
        d^2(x,A+B)&\leq d^2(x,\textup{vert}(P))\\
        &\leq\left(\frac{2d(A)}{2}\right)^2+\left(\frac{2d(B)}{2}\right)^2\\
        &=d^2(A)+d^2(B).
    \end{split}
\end{equation*}
It follows that $d^2(A+B)\leq d^2(A)+d^2(B)$. This proves the theorem.
\end{proof}

\section{Open questions}\label{open_questions}

We conclude with some suggestions for further research.

A \textit{centrally symmetric} convex body $K\subseteq \mathbb{R}^n$ is a convex body for which $x\in K$ implies $-x\in K$. We define the generalized Hausdorff distance to convex hull by
\begin{equation*}
    d^{(K)}(A):=\inf\{r\geq 0: \textup{conv}(A)\subseteq A+rK\}.
\end{equation*}
Equivalently
\begin{equation*}
    d^{(K)}(A)=\sup_{a\in\textup{conv}(A)}\inf_{x\in A}\|x-a\|_{K},
\end{equation*}
where $\|\cdot\|_{K}$ is the norm induced by $K$. It is interesting to ask if Theorem \ref{dyn_farkhi_2_dimension} holds in this general setting.

\begin{question}
    Let $K$ be a centrally symmetric convex body in $\mathbb{R}^2$. What is the best upper bound for $d^{(K)}(A+B)$ in terms of $d^{(K)}(A)$ and $d^{(K)}(B)$ in the same sense given in Theorem \ref{dyn_farkhi_2_dimension}?
\end{question}
Due to Lemma \ref{sumset_convexhull_two_dimension} the above question can be resolved by finding a formula for $d^{(K)}(\textup{vert}(P))$, where $P$ is a parallelogram.

Recall from the introduction that the Dyn-Farkhi conjecture has not been resolved (when $n\geq3$) for the case $A=B$. One possible approach to this would be to extend Lemma \ref{sumset_convexhull_two_dimension} to a statement about sumsets in $\mathbb{R}^n$. Such a statement would be interesting even on its own.

\begin{question}
    For an integer $n\geq3$ let $A$ and $B$ be full dimensional compact sets in $\mathbb{R}^n$. Can we generalize Lemma \ref{sumset_convexhull_two_dimension} to achieve a characterization for the sumset $\textup{conv}(A+B)$?
\end{question}

To prove Lemma \ref{sumset_convexhull_two_dimension} we relied on the fact that the geometry of the boundary of the convex hull of a 2 dimensional set is relatively simple (we are dealing with points and intervals). But, the complication when trying to generalize to $n\geq3$ is that the faces of the convex hull can now be $2$ dimensional and higher with a much more complicated structure, so it seems unlikely that we can achieve a nice characterization with parallelograms.

\newpage
\bibliography{hausdorff_distance}
\bibliographystyle{abbrv}

\end{document}